\documentclass[reqno,a4paper]{amsart}
\usepackage{amssymb}
\usepackage[hypertex]{hyperref}
\date{This manuscript was completed on 21 January 2006}
\date{}
\allowdisplaybreaks[4]
\theoremstyle{plain}
\newtheorem{thm}{Theorem}
\newtheorem{lem}{Lemma}

\theoremstyle{remark}
\newtheorem{rem}{Remark}
\newcommand{\abs}[1]{\left\vert#1\right\vert}
\DeclareMathOperator{\td}{d\mspace{-2mu}}
\newcommand{\tn}{\mathbb{N}}

\begin{document}

\title[A class of completely monotonic functions and applications]
{A class of completely monotonic functions involving divided differences of the psi and polygamma functions and some applications}

\author[F. Qi]{Feng Qi}
\address[F. Qi]{Research Institute of Mathematical Inequality Theory, Henan Polytechnic University, Jiaozuo City, Henan Province, 454010, China}
\email{\href{mailto: F. Qi <qifeng618@gmail.com>}{qifeng618@gmail.com}, \href{mailto: F. Qi <qifeng618@hotmail.com>}{qifeng618@hotmail.com}, \href{mailto: F. Qi <qifeng618@qq.com>}{qifeng618@qq.com}}
\urladdr{\url{http://qifeng618.spaces.live.com}}

\author[B.-N. Guo]{Bai-Ni Guo}
\address[B.-N. Guo]{School of Mathematics and Informatics, Henan Polytechnic University, Jiaozuo City, Henan Province, 454010, China}
\email{\href{mailto: B.-N. Guo <bai.ni.guo@gmail.com>}{bai.ni.guo@gmail.com}, \href{mailto: B.-N. Guo <bai.ni.guo@hotmail.com>}{bai.ni.guo@hotmail.com}}

\begin{abstract}
A class of functions involving the divided differences of the psi function and the polygamma functions and originating from Kershaw's double inequality are proved to be completely monotonic. As applications of these results, the monotonicity and convexity of a function involving ratio of two gamma functions and originating from establishment of the best upper and lower bounds in Kershaw's double inequality are derived, two sharp double inequalities involving ratios of double factorials are recovered, the probability integral or error function is estimated, a double inequality for ratio of the volumes of the unit balls in $\mathbb{R}^{n-1}$ and $\mathbb{R}^n$ respectively is deduced, and a symmetrical upper and lower bounds for the gamma function in terms of the psi function is generalized.
\end{abstract}



\thanks{This paper was typeset using \AmS-\LaTeX}

\maketitle

\section{Introduction}

Recall~\cite[Chapter~XIII]{mpf-1993} and~\cite[Chapter~IV]{widder} that a function $f$ is said to be completely monotonic on an interval $I$ if $f$ has derivatives of all orders on $I$ and
\begin{equation}
(-1)^{n}f^{(n)}(x)\ge0
\end{equation}
for $x \in I$ and $n \geq0$. The famous Bernstein's Theorem in~\cite[p.~160, Theorem~12a]{widder} states that a function $f$ is completely monotonic on $[0,\infty)$ if and only if
\begin{equation}\label{converge}
f(x)=\int_0^\infty e^{-xs}\td\mu(s),
\end{equation}
where $\mu$ is a nonnegative measure on $[0,\infty)$ such that the integral~\eqref{converge} converges for all $x>0$. This expresses that a completely monotonic function $f$ on $[0,\infty)$ is a Laplace transform of the measure $\mu$.
\par
Recall also~\cite{Atanassov, minus-one, auscm-rgmia} that a positive function $f$ is called logarithmically completely monotonic on an interval $I$ if $f$ has derivatives of all orders on $I$ and its logarithm $\ln f$ satisfies
\begin{equation}
(-1)^k[\ln f(x)]^{(k)}\ge0
\end{equation}
for all $k\in\tn$ on $I$. It was proved explicitly in~\cite{CBerg, compmon2, absolute-mon.tex, minus-one, schur-complete} by different approaches that any logarithmically completely monotonic function must be completely monotonic, but not conversely. It was pointed out in~\cite[Theorem~1.1]{CBerg} and~\cite{grin-ismail,auscm} that the logarithmically completely monotonic functions on $[0,\infty)$ are those completely monotonic functions on $[0,\infty)$ for which the representing measure $\mu$ in~\eqref{converge} is infinitely divisible in the convolution sense: For each $n\in\mathbb{N}$ there exists a positive measure $\nu$ on $[0,\infty)$ with $n$-th convolution power equal to $\mu$.
\par
It is well-known that the classical Euler gamma function
\begin{equation}
\Gamma(x)=\int^\infty_0t^{x-1} e^{-t}\td t
\end{equation}
for $x>0$, its logarithmic derivative, denoted by $\psi(x)=\frac{\Gamma'(x)}{\Gamma(x)}$, and the polygamma functions $\psi^{(i)}(x)$ for $i\in\mathbb{N}$ are several of the most important special functions and have much extensive applications in many branches such as statistics, probability, number theory, theory of $0$-$1$ matrices, graph theory, combinatorics, physics, engineering, and other mathematical sciences.
\par
The ratio $\frac{\Gamma(x+p)}{\Gamma(x+q)}$ for $x+p>0$ and $x+q>0$ of two gamma functions, called Wallis function or ratio in the literature, has been investigated since $1948$ in~\cite{wendel} at least. Now there exist a lot of conclusions on Wallis ratio, its variants, generalizations and applications, for example,~\cite{ball-volume-rn, Bustoz-and-Ismail, dutka, gaut, grin-ismail, laforgia-mc-1984, ratio-gamma-polynomial.tex, ratio-gamma-polynomial.tex-rgmia, sandor-gamma-3.tex, sandor-gamma-3.tex-rgmia, notes-best-simple-equiv.tex-RGMIA, notes-best-simple-equiv.tex, sandor-gamma-2.tex, sandor-gamma-2.tex-rgmia, gamma-batir.tex, sandor-gamma-3-note.tex-final, sandor-gamma-3-note.tex} and related references therein.
\par
In~\cite{kershaw}, D. Kershaw proved a double inequality
\begin{equation}\label{gki1}
\bigg(x+\frac{s}2\bigg)^{1-s}< \frac{\Gamma(x+1)}{\Gamma(x+s)}
<\Bigg(x-\frac12+\sqrt{s+\frac14}\,\Bigg)^{1-s}
\end{equation}
for $0<s<1$ and $x\ge1$. It is clear that the inequality~\eqref{gki1} can be
rearranged as
\begin{equation}\label{gki1-rewr}
\frac{s}2<\bigg[\frac{\Gamma(x+1)}{\Gamma(x+s)}\bigg]^{1/({1-s})}-x
<\sqrt{s+\frac14}\,-\frac12.
\end{equation}
This suggests us to introduce a function
\begin{equation}
z_{s,t}(x)=\begin{cases}
\bigg[\dfrac{\Gamma(x+t)}{\Gamma(x+s)}\bigg]^{1/(t-s)}-x,&s\ne t\\
e^{\psi(x+s)}-x,&s=t
\end{cases}
\end{equation}
on $x\in(-\alpha,\infty)$ for real numbers $s$ and $t$ and $\alpha=\min\{s,t\}$.
\par
In~\cite{201-05-JIPAM, egp, notes-best-simple-equiv.tex-RGMIA, notes-best-simple-equiv.tex, notes-best-simple-open.tex, notes-best-simple.tex, notes-best.tex, notes-best.tex-rgmia}, the monotonic and convex properties of $z_{s,t}(x)$ were established by using Laplace transform and other complicated techniques. Their basic calculation is as follows:
\begin{gather}
z'_{s,t}(x)=[z_{s,t}(x)+x]\frac{\psi(x+t)-\psi(x+s)}{t-s}-1, \label{zi}\\
z''_{s,t}(x)=[z_{s,t}(x)+x]\bigg\{\bigg[\frac{\psi(x+t)
-\psi(x+s)}{t-s}\bigg]^2 +\frac{\psi'(x+t)-\psi'(x+s)}{t-s}\bigg\}\label{z''-delta}\\
=\frac{z_{s,t}(x)+x}{(t-s)^2}\bigg\{{[\psi(x+t) -\psi(x+s)]^2}
+(t-s)[\psi'(x+t)-\psi'(x+s)]\bigg\}.\label{z''-theta}
\end{gather}
This further suggests us to consider the following two functions:
\begin{equation}\label{Delta-dfn}
\Delta_{s,t}(x)=\begin{cases}\bigg[\dfrac{\psi(x+t) -\psi(x+s)}{t-s}\bigg]^2
+\dfrac{\psi'(x+t)-\psi'(x+s)}{t-s},&s\ne t\\
[\psi'(x+s)]^2+\psi''(x+s),&s=t
\end{cases}
\end{equation}
and
\begin{equation}\label{Theta-dfn}
\Theta_{s,t}(x)=[{\psi(x+t) -\psi(x+s)}]^2 +(t-s)[{\psi'(x+t)-\psi'(x+s)}]
\end{equation}
on $x\in(-\alpha,\infty)$ for real numbers $s$ and $t$ and $\alpha=\min\{s,t\}$.
\par
In~\cite[p.~208]{forum-alzer}, \cite[Lemma~1.1]{batir-new} and~\cite[Lemma~1.1]{batir-new-rgmia}, the inequality
\begin{equation}\label{positivity}
\Delta_{0,0}(x)=[\psi'(x)]^2+\psi''(x)>0
\end{equation}
on $(0,\infty)$ was verified. In~\cite{alzer-grinshpan, batir-new, batir-new-rgmia}, this inequality was applied to provide some symmetrical upper and lower bounds for $\Gamma(x)$ in terms of $\psi(x)$ as follows:
\begin{equation}\label{batir-alzer-grinshpan-ineq}
\exp\bigl\{\alpha\bigl[e^{\psi(x)}(\psi(x)-1)+1\bigr]\bigr\}\le \frac{\Gamma(x)}{\Gamma(x^\ast)} \le\exp\bigl\{\beta\bigl[e^{\psi(x)}(\psi(x)-1)+1\bigr]\bigr\},
\end{equation}
where $x^\ast=1.4616\dotsm$ denotes the only positive zero of $\psi(x)$, and $\alpha$ and $\beta$ are real constants.
\par
The first aim of this paper is to present the completely monotonic property of the functions $\Delta_{s,t}(x)$ and $\Theta_{s,t}(x)$ on $(-\alpha,\infty)$, which implies the positivity of the function $\Delta_{0,0}(x)$ defined by~\eqref{positivity}.

\begin{thm}\label{byproduct-1}
The functions $\Delta_{s,t}(x)$ for $\abs{t-s}<1$ and $-\Delta_{s,t}(x)$ for $\abs{t-s}>1$ are completely monotonic on $x\in(-\alpha,\infty)$. So are the functions $\Theta_{s,t}(x)$ for $\abs{t-s}<1$ and $-\Theta_{s,t}(x)$ for $\abs{t-s}>1$ on $x\in(-\alpha,\infty)$.
\end{thm}

The second aim of this paper is, by making use of Theorem~\ref{byproduct-1}, to provide an alternative proof for the monotonicity and convexity of the function $z_{s,t}(x)$, which is quoted as follows.

\begin{thm}[\cite{201-05-JIPAM, egp, notes-best-simple-open.tex,
notes-best-simple.tex, notes-best.tex, notes-best.tex-rgmia}]\label{egp-mon}
The function $z_{s,t}(x)$ in
$(-\alpha,\infty)$ is either convex and decreasing for $\abs{t-s}<1$ or
concave and increasing for $\abs{t-s}>1$.
\end{thm}

It is well-known~\cite{WallisCosineFormula} that Wallis cosine or sine formula is
\begin{multline}\label{cosine-formula}
\int_0^{\pi/2}\sin^nx\td x=\int_0^{\pi/2}\cos^nx\td x\\
=\frac{\sqrt\pi\,\Gamma((n+1)/2)}{n\Gamma({n}/2)}=\begin{cases}
\dfrac{\pi}2\cdot\dfrac{(n-1)!!}{n!!}&\text{for $n$ even},\\[1em]
\dfrac{(n-1)!!}{n!!}&\text{for $n$ odd},
\end{cases}
\end{multline}
where $n!!$ denotes the double factorial. It has been estimated by many mathematicians and a lot of inequalities were established in, for example,~\cite{wallis-cao, wallis3, wallis-indon, wallis-tamk, wallis2, chenwallis, chenwallis-rgmia, wallis-gaz, kaz-wallis, Koumandos-PAMS-06, mia-qi-cui-xu-99, waston, zhao-wu} and related references therein.
\par
The third aim of this paper is, by utilizing Theorem~\ref{egp-mon}, to prove two sharp double inequalities relating to Wallis cosine or sine formula~\eqref{cosine-formula} and to bound the probability integral or error function as follows.

\begin{thm}\label{wallis-ineq-thm}
For $n\in\mathbb{N}$,
\begin{gather}\label{best bounds Wallis}
\frac{1}{\sqrt{\pi(n+{4}/{\pi}-1)}}\le\frac{(2n-1)!!}{(2n)!!}
<\frac{1}{\sqrt{\pi(n+1/4)}},\\\label{Wallis-type ineq}
\frac{\sqrt{\pi}}{2\sqrt{n+{9\pi}/{16}-1}}\le\frac{(2n)!!}{(2n+1)!!}
<\frac{\sqrt{\pi}}{2\sqrt{n+3/4}}
\end{gather}
and
\begin{equation}\label{17inequal}
\frac{\sqrt{\pi}}{\sqrt{1+(9\pi/16-1)/n}}
\le\int_{-\sqrt{n}}^{\sqrt{n}}e^{-x^2} \td x
<\frac{\sqrt{\pi}}{\sqrt{1-3/{(4n)}}}.
\end{equation}
In particular, taking $n\to\infty$ in~\eqref{17inequal} leads to
\begin{equation}
\int_{-\infty}^{\infty}e^{-x^2}\td x=\sqrt{\pi}\,.
\end{equation}
The constants $\frac{4}{\pi}-1$ and $\frac14$ in~\eqref{best bounds Wallis}
and the constants $\frac{9\pi}{16}-1$ and $\frac34$ in~\eqref{Wallis-type
ineq} are the best possible.
\end{thm}

Let
\begin{equation}
\Omega_n=\frac{\pi^{n/2}}{\Gamma(1+n/2)}
\end{equation}
be the volume of the unit ball on $\mathbb{R}^n$. The fourth aim of this paper is, by employing Theorem~\ref{egp-mon}, to recover a double inequality for ratio of the volumes of the unit balls in $\mathbb{R}^{n-1}$ and $\mathbb{R}^n$ respectively as follows.

\begin{thm}[{\cite[Theorem~2]{ball-volume-rn}}]\label{alzer-ball-thm-2}
For $n\in\mathbb{N}$, the inequality
\begin{equation}\label{alzer-ball-3.8}
\sqrt{\frac{n+A}{2\pi}}<\frac{\Omega_{n-1}}{\Omega_n} \le\sqrt{\frac{n+B}{2\pi}}
\end{equation}
holds if and only if $A\le\frac12$ and $B\ge\frac\pi2-1$.
\end{thm}

The final aim of this paper is, by using Theorem~\ref{byproduct-1}, to generalize the inequality~\eqref{batir-alzer-grinshpan-ineq} to a monotonic property as follows.

\begin{thm}\label{open-solve}
For real numbers $s$ and $t$, $\alpha=\min\{s,t\}$ and $c\in(-\alpha,\infty)$, let
\begin{equation}
g_{s,t}(x)=\begin{cases}\displaystyle
\frac1{t-s} \int_c^x\ln\biggl[\frac{\Gamma(u+t)}{\Gamma(u+s)} \frac{\Gamma(c+s)}{\Gamma(c+t)}\biggr]\td u,&s\ne t\\[1em]
\displaystyle
\int_c^x[\psi(u+s)-\psi(c+s)]\td u,&s=t
\end{cases}
\end{equation}
on $x\in(-\alpha,\infty)$. Then the function
\begin{equation}
f_{s,t}(x)=\begin{cases}\displaystyle
\frac{g_{s,t}(x)}{[g'_{s,t}(x)-1]\exp[g'_{s,t}(x)]+1},&x\ne c\\[1em]
\dfrac1{g''_{s,t}(c)},&x=c
\end{cases}
\end{equation}
on $x\in(-\alpha,\infty)$ is decreasing for $|s-t|<1$ and increasing for $|s-t|>1$.
\end{thm}

\begin{rem}
If taking $c=x^\ast$, then the case $s=t$ in Theorem~\ref{open-solve} becomes~\cite[Theorem~4.3]{alzer-grinshpan}:
For $0<a<b\le\infty$ and $x\in(a,b)$, the inequality~\eqref{batir-alzer-grinshpan-ineq} holds with the best possible constant factors
\begin{equation}
\alpha=\begin{cases}
Q(b),&b<\infty\\1,&b=\infty
\end{cases}\quad\text{and}\quad \beta=Q(a),
\end{equation}
where
\begin{equation}
Q(x)=\begin{cases}
\dfrac{\ln\Gamma(x)-\ln\Gamma(x^\ast)} {e^{\psi(x)}[\psi(x)-1]+1}, &x\ne x^\ast;\\[1em]
\dfrac1{\psi'(x^\ast)},&x=x^\ast.
\end{cases}
\end{equation}
\end{rem}

\section{Lemmas}

In order to prove our theorems, the following lemmas are necessary.

\begin{lem}\label{elbert-laforgia-lem}
Let $f(x)$ be defined in an infinite interval $I$. If
$$
\lim_{x\to\infty}f(x)=\delta\quad \text{and}\quad f(x)-f(x+\varepsilon)>0
$$
for some given $\varepsilon>0$, then $f(x)>\delta$ on $I$.
\end{lem}

\begin{proof}
By induction, for any $x\in I$, we have
\begin{equation*}
f(x)>f(x+\varepsilon)>f(x+2\varepsilon)>\dotsm>f(x+k\varepsilon)\to\delta
\end{equation*}
as $k\to\infty$. The proof of Lemma \ref{elbert-laforgia-lem} is complete.
\end{proof}

\begin{lem}[\cite{abram}]
For any positive integer $n\in\tn$ and $x>0$,
\begin{gather}\label{psi-int-expr}
\psi(x)=\ln x+ \int_{0}^{\infty}\biggl[\frac1u-\frac1{1-e^{-u}}\biggr]e^{-xu}\td u,\\
\label{psin}
\psi ^{(n)}(x)=(-1)^{n+1}\int_{0}^{\infty}\frac{u^{n}}{1-e^{-u}}e^{-xu}\td u,\\
\label{psisymp4}
\psi^{(n-1)}(x+1)=\psi^{(n-1)}(x)+\frac{(-1)^{n-1}(n-1)!}{x^n},\\
\label{psi-ineq}
\ln x-\frac1x<\psi(x)<\ln x-\frac1{2x}.
\end{gather}
As $x\to\infty$,
\begin{equation}
\label{psi1-asym}
\psi'(x)\sim\frac1x+\frac1{2x^2}+\dotsm.
\end{equation}
\end{lem}

\begin{lem}[\cite{ingamma, ingamma-rgmia}]
For $s>r>0$,
\begin{equation}\label{qi-gam-ineq}
\exp\left[(s-r)\psi(s)\right]>\frac{\Gamma(s)}{\Gamma(r)} >\exp\left[(s-r)\psi(r)\right].
\end{equation}
\end{lem}

\begin{lem}[\cite{widder}]\label{comp-func-product}
A product of finite completely monotonic functions is also completely monotonic.
\end{lem}

\section{Proofs of theorems}

Now we are in a position to prove our theorems.

\begin{proof}[Proof of Theorem \ref{byproduct-1}]
Direct computation and utilization of~\eqref{psisymp4} gives
\begin{gather}
\Theta_{s,t}(x)-\Theta_{s,t}(x+1)
=\big\{[\psi(x+t)+\psi(x+t+1)]-[\psi(x+s)+\psi(x+s+1)]\big\}\notag\\
\begin{aligned}\notag
&\times\big\{[\psi(x+t)-\psi(x+t+1)]-[\psi(x+s)-\psi(x+s+1)]\big\}\\
&+(t-s)\big\{[\psi'(x+t)-\psi'(x+t+1)]-[\psi'(x+s)-\psi'(x+s+1)]\big\}
\end{aligned}\\
\begin{aligned}\label{lambda-dfn}
&=\bigg\{\frac{[\psi(x+t+1)+\psi(x+t)]-[\psi(x+s+1)+\psi(x+s)]}{t-s}\\
&\quad-\frac{2x+s+t}{(x+s)(x+t)}\bigg\}\frac{(t-s)^2}{(x+s)(x+t)}\\
&\triangleq\Lambda_{s,t}(x)\frac{(t-s)^2}{(x+s)(x+t)}
\end{aligned}
\end{gather}
and
\begin{align*}
\Lambda_{s,t}(x)-\Lambda_{s,t}(x+1)
&=\frac1{t-s}\bigg(\frac1{x+s}+\frac1{x+s+1}-\frac1{x+t}-\frac1{x+t+1}\bigg)\\
&\quad-\frac{2x^2+2(s+t+1)x+s^2+t^2+s+t}{(x+s)(x+s+1)(x+t)(x+t+1)}\\
&=\frac{1-(s-t)^2}{(x+s)(x+s+1)(x+t)(x+t+1)}.
\end{align*}
\par
Since $\lim_{x\to\infty}\Lambda_{s,t}^{(i)}(x)=0$ for any nonnegative integer
$i$ by~\eqref{psi-int-expr} and~\eqref{psin} and the function
$$
\frac{\Lambda_{s,t}(x)-\Lambda_{s,t}(x+1)}{1-(s-t)^2}
$$
is completely monotonic by Lemma~\ref{comp-func-product}, that is,
$$
(-1)^i\frac{[\Lambda_{s,t}(x)-\Lambda_{s,t}(x+1)]^{(i)}}{1-(s-t)^2}
=\frac{(-1)^i\Lambda_{s,t}^{(i)}(x)-(-1)^i\Lambda_{s,t}^{(i)}(x+1)}{1-(s-t)^2}\ge0,
$$
on $(-\alpha,\infty)$, then
$
\frac{(-1)^i\Lambda_{s,t}^{(i)}(x)}{1-(s-t)^2}\ge0
$
follows from Lemma~\ref{elbert-laforgia-lem}. This means the function
$\frac{\Lambda_{s,t}(x)}{1-(s-t)^2}$ is completely monotonic in
$(-\alpha,\infty)$.
\par
Since the function $\frac{(t-s)^2}{(x+s)(x+t)}$ is completely monotonic,
then the function
$$
\frac{\Theta_{s,t}(x)-\Theta_{s,t}(x+1)}{1-(s-t)^2}
$$
is completely monotonic on $(-\alpha,\infty)$ by considering~\eqref{lambda-dfn} and Lemma~\ref{comp-func-product}, which is equivalent to
$$
(-1)^k\bigg[\frac{\Theta_{s,t}(x)-\Theta_{s,t}(x+1)}{1-(s-t)^2}\bigg]^{(k)}
=\frac{(-1)^k\Theta_{s,t}^{(k)}(x)-(-1)^k\Theta_{s,t}^{(k)}(x+1)}{1-(s-t)^2}\ge0
$$
for nonnegative integer $k$. Further, from
$\lim_{x\to\infty}\Theta_{s,t}^{(k)}(x)=0$ for nonnegative integer $k$, which
can be deduced by utilizing~\eqref{psi-int-expr} and~\eqref{psin}, and Lemma
\ref{elbert-laforgia-lem}, it is concluded that
$\frac{(-1)^k\Theta_{s,t}^{(k)}(x)}{1-(s-t)^2}\ge0$ for any nonnegative
integer $k$. This implies $(-1)^k\Theta_{s,t}^{(k)}(x)\gtreqless0$ if and only
if $\abs{t-s}\lessgtr1$. Therefore, the functions $\Theta_{s,t}(x)$ for
$\abs{t-s}<1$ and $-\Theta_{s,t}(x)$ for $\abs{t-s}>1$ are completely
monotonic on $(-\alpha,\infty)$.
\par
Since $\Theta_{s,t}(x)=(t-s)^2\Delta_{s,t}(x)$, the function $\Delta_{s,t}(x)$
has the same monotonicity property as $\Theta_{s,t}(x)$ on $(-\alpha,\infty)$.
The proof of Theorem \ref{byproduct-1} is complete.
\end{proof}

\begin{proof}[Proof of Theorem \ref{egp-mon}]
It is clear from~\eqref{z''-delta} and
\eqref{z''-theta} that
\begin{equation}\label{z''}
z''_{s,t}(x)=[z_{s,t}(x)+x]\Delta_{s,t}(x)
=\frac{{z_{s,t}(x)+x}}{(t-s)^2}\Theta_{s,t}(x)
\end{equation}
for $t\ne s$. By Theorem \ref{byproduct-1}, it is easy to see that
$\Theta_{s,t}(x)\gtreqless0$ and $\Delta_{s,t}(x)\gtreqless0$ in
$(-\alpha,\infty)$ if and only if $\abs{t-s}\lessgtr1$. Then
$z''_{s,t}(x)\gtreqless0$ for $\abs{t-s}\lessgtr1$ follows from formula
\eqref{z''}. The convexity and concavity of the function $z_{s,t}(x)$ is
proved.
\par
The inequality~\eqref{qi-gam-ineq} is equivalent to
\begin{equation*}
\max\big\{e^{\psi(s)},e^{\psi(r)}\big\}
>\bigg[\frac{\Gamma(s)}{\Gamma(r)}\bigg]^{1/(s-r)}
>\min\big\{e^{\psi(s)},e^{\psi(r)}\big\}
\end{equation*}
for any positive numbers $s>0$ and $t>0$. This implies
\begin{equation}\label{z'>}
\begin{split}
z'_{s,t}(x)&=\bigg[\dfrac{\Gamma(x+t)}{\Gamma(x+s)}\bigg]^{1/(t-s)} \frac{\psi(x+t)-\psi(x+s)}{t-s}-1\\
&<e^{\psi(x+t)}\frac{\psi(x+t)-\psi(x+s)}{t-s}-1\\
&=e^{\psi(x+t)}\psi'(x+\xi)-1\\
&<\psi'(x+t)e^{\psi(x+t)}-1
\end{split}
\end{equation}
and
\begin{equation}\label{z'<}
\begin{split}
z'_{s,t}(x)&>e^{\psi(x+s)}\frac{\psi(x+t)-\psi(x+s)}{t-s}-1\\*
&=e^{\psi(x+s)}\psi'(x+\xi)-1\\
&>\psi'(x+s)e^{\psi(x+s)}-1
\end{split}
\end{equation}
if assuming $t>s>0$ without loss of generality, where $\xi\in(s,t)$.
\par
By the inequality~\eqref{psi-ineq}, we obtain
\begin{equation}
x\psi'(x)e^{-1/x}<\psi'(x)e^{\psi(x)}<x\psi'(x)e^{-1/2x}
\end{equation}
for $x>0$. Using the asymptotic representation~\eqref{psi1-asym} yields
\begin{equation}
\lim_{x\to\infty}\big[x\psi'(x)e^{-1/x}\big]=1\quad \text{and} \quad
\lim_{x\to\infty}\big[x\psi'(x)e^{-1/2x}\big]=1.
\end{equation}
Hence,
\begin{equation}\label{lim=1}
\lim_{x\to\infty}\big[\psi'(x)e^{\psi(x)}\big]=1.
\end{equation}
Combining~\eqref{lim=1} with~\eqref{z'>} and~\eqref{z'<} leads to
\begin{gather*}
\lim_{x\to\infty}z'_{s,t}(x)\le
\lim_{x\to\infty}\big[\psi'(x+t)e^{\psi(x+t)}\big]-1
=\lim_{x+t\to\infty}\big[\psi'(x+t)e^{\psi(x+t)}\big]-1=0\intertext{and}
\lim_{x\to\infty}z'_{s,t}(x)\ge
\lim_{x\to\infty}\big[\psi'(x+s)e^{\psi(x+s)}\big]-1
=\lim_{x+s\to\infty}\big[\psi'(x+s)e^{\psi(x+s)}\big]-1=0.
\end{gather*}
Thus, it is concluded that $\lim_{x\to\infty}z'_{s,t}(x)=0$.
\par
Since $z''_{s,t}(x)\gtreqless0$ on $x\in(-\alpha,\infty)$ for
$\abs{t-s}\lessgtr1$, then the function $z'_{s,t}(x)$ is increasing/decreasing
on $x\in(-\alpha,\infty)$ for $\abs{t-s}\lessgtr1$. Thus, it follows that
$z'_{s,t}(x)\lesseqgtr0$ and $z_{s,t}(x)$ is decreasing/increasing in
$x\in(-\alpha,\infty)$ for $\abs{t-s}\lessgtr1$. The monotonicity of the
function $z_{s,t}(x)$ is proved.
\end{proof}

\begin{proof}[The second proof of convexity of $z_{s,t}(x)$]
It is sufficient to show the function
\begin{equation}
\Phi_{s,t}(x)=
\begin{cases}
\dfrac{\psi(x+s)-\psi(x+t)}{s-t} \biggl[\dfrac{\Gamma(x+s)}{\Gamma(x+t)}\biggr]^{1/(s-t)},& s\ne t\\[1em]
\psi'(x+s)e^{\psi(x+s)},& s=t
\end{cases}
\end{equation}
on $(-\alpha,\infty)$ is increasing for $\abs{t-s}<1$ and decreasing for $\abs{t-s}>1$.
\par
Straightforward calculation yields
\begin{equation*}
\ln\Phi_{s,t}(x)=
\begin{cases}
\ln\dfrac{\psi(x+s)-\psi(x+t)}{s-t} +\dfrac{\ln\Gamma(x+s)-\ln\Gamma(x+t)}{s-t},& s\ne t\\
\ln\psi'(x+s)+\psi(x+s),& s=t
\end{cases}
\end{equation*}
and
\begin{align*}
[\ln\Phi_{s,t}(x)]'&=
\begin{cases}
\dfrac{\psi'(x+s)-\psi'(x+t)}{\psi(x+s)-\psi(x+t)} +\dfrac{\psi(x+s)-\psi(x+t)}{s-t},& s\ne t\\[1em]
\dfrac{\psi''(x+s)}{\psi'(x+s)}+\psi'(x+s),& s=t
\end{cases}\\
&=\begin{cases}
\dfrac{s-t}{\psi(x+s)+\psi(x+t)}\Delta_{s,t}(x),& s\ne t\\[1em]
\dfrac1{\psi'(x+s)}\Delta_{s,s}(x),& s=t.
\end{cases}
\end{align*}
In virtue of Theorem~\ref{byproduct-1}, it is concluded that
\begin{equation*}
[\ln\Phi_{s,t}(x)]'
\begin{cases}
>0,&\text{if $\abs{t-s}<1$,}\\
<0,&\text{if $\abs{t-s}>1$.}
\end{cases}
\end{equation*}
The second proof of convexity of $z_{s,t}(x)$ is complete.
\end{proof}

\begin{rem}
In~\cite{batir-interest, batir-interest-rgmia, batir-new, batir-new-rgmia, egp}, the inequalities
\begin{equation}\label{psi-psi-e}
\psi'(x)e^{\psi(x)}<1,\quad x>0
\end{equation}
and~\eqref{positivity} were proved and used to construct many inequalities for bounding the gamma function $\Gamma(x)$, the psi function $\psi$ and the trigamma function $\psi'(x)$ such as~\eqref{batir-alzer-grinshpan-ineq}.
\par
In~\cite{egp}, as a corollary of Theorem \ref{egp-mon}, the following inequality was deduced:
\begin{equation}\label{psipsi}
\bigg[\dfrac{\Gamma(x+t)}{\Gamma(x+s)}\bigg]^{1/(t-s)} <\frac{t-s}{\psi(x+t)-\psi(x+s)}
\end{equation}
holds for $0<\abs{t-s}<1$ and with reversed sign if $\abs{t-s}>1$.
\par
The second proof of convexity of $z_{s,t}(x)$ implies that inequalities~\eqref{positivity},~\eqref{psi-psi-e} and~\eqref{psipsi} are equivalent to each other.
\end{rem}

\begin{rem}
It is conjectured that the function $\Phi_{s,t}(x)$ for $\abs{t-s}>1$ and its reciprocal for $\abs{t-s}<1$ are logarithmically completely monotonic on $x\in(-\alpha,\infty)$, which modified an open problem posed in~\cite{notes-best.tex, notes-best.tex-rgmia}.
\end{rem}

\begin{proof}[Proof of Theorem~\ref{wallis-ineq-thm}]
Let the sequence $\theta_1(n)$ be defined by
\begin{equation}
\frac{(2n-1)!!}{(2n)!!} =\frac1{\sqrt{\pi(n+\theta_1(n))}}
\end{equation}
for $n\in\mathbb{N}$. In order to obtain the inequality~\eqref{best bounds
Wallis}, it is sufficient to show
\begin{equation}\label{theta-1}
\frac14<\theta_1(n)\le\frac{4}{\pi}-1
\end{equation}
for all $n\in\mathbb{N}$. Indeed, formula~\eqref{cosine-formula} and Theorem
\ref{egp-mon} implies
\begin{equation}\label{theta-dfn-1}
\theta_1(x)=\biggl[\frac{\Gamma(x+1)}{\Gamma(x+1/2)}\biggr]^2-x=z_{1/2,1}(x)
\end{equation}
 is convex and decreasing in
$\left(-\frac12,\infty\right)$, and then the sharp double inequality
\eqref{theta-1} can be deduced by observing that $\theta_1(1)=\frac4\pi-1$ and
$\lim_{x\to\infty}\theta_1(x)=\frac14$.
\par
Let the sequence $\theta_2(n)$ be defined by
\begin{equation}
\frac{(2n)!!}{(2n+1)!!} =\frac{\sqrt{\pi}}{2\sqrt{n+\theta_2(n)}}
\end{equation}
for $n\in\mathbb{N}$. In order to obtain the inequality~\eqref{Wallis-type ineq},
it is sufficient to show
\begin{equation}\label{theta-2}
\frac34<\theta_2(n)\le\frac{9\pi}{16}-1
\end{equation}
for all $n\in\mathbb{N}$. Indeed, formula~\eqref{cosine-formula} and Theorem
\ref{egp-mon} implies
\begin{equation}\label{theta-dfn-2}
\theta_2(x)=\left[\frac{\Gamma(x+3/2)}{\Gamma(x+1)}\right]^2-x=z_{1,3/2}(x)
\end{equation}
 is convex and decreasing in
$(-1,\infty)$, and then the sharp double inequality~\eqref{theta-2} can be
deduced by observing that $\theta_2(1)=\frac{9\pi}{16}-1$ and
$\lim_{x\to\infty}\theta_2(x)=\frac34$.
\par
The rest is the same as the proof of~\cite[Theorem~1.2]{wallis-cao}.
\end{proof}

\begin{proof}[Proof of Theorem~\ref{alzer-ball-thm-2}]
The inequality~\eqref{alzer-ball-3.8} can be rearranged as
\begin{equation}\label{a/2-b/2}
\frac{A}2<z_{1,1/2}\left(\frac{n}2\right)\le\frac{B}2
\end{equation}
for $n\in\mathbb{N}$. Since $z_{1,1/2}(x)$ is decreasing in
$\left(-\frac12,1\right)$ by Theorem~\ref{egp-mon}, considering
$z_{1,1/2}\left(\frac12\right)=\frac\pi4-\frac12$ and
$\lim_{x\to\infty}z_{1,1/2}(x)=\frac14$, then the inequality~\eqref{a/2-b/2} is
concluded.
\end{proof}

\begin{proof}[Proof of Theorem~\ref{open-solve}]
Straightforward computation yields
\begin{equation}
g'_{s,t}(x)=\begin{cases}\displaystyle
\frac1{t-s} \ln\biggl[\frac{\Gamma(x+t)}{\Gamma(x+s)} \frac{\Gamma(c+s)}{\Gamma(c+t)}\biggr],&s\ne t;\\[1em]
\displaystyle
\psi(x+s)-\psi(c+s),&s=t.
\end{cases}
\end{equation}
It is clear that $g_{s,t}(c)=g_{s,t}'(c)=0$ and
\begin{equation*}
g''_{s,t}(x)=\begin{cases}\displaystyle
\frac{\psi(x+t)-\psi(x+s)}{t-s}>0,&s\ne t;\\
\psi'(x+s)>0,&s=t.
\end{cases}
\end{equation*}
This leads to $(x-c)g'_{s,t}(x)>0$ and $g_{s,t}(x)>0$ for $x\ne c$.
\par
Differentiation yields
\begin{equation*}
f_{s,t}'(x)=\frac{g'_{s,t}(x)h_{s,t}(x)}{\{[g'_{s,t}(x)-1]\exp[g'_{s,t}(x)]+1\}^2},
\end{equation*}
where
\begin{equation*}
h_{s,t}(x)=[g'_{s,t}(x)-g_{s,t}(x)g''_{s,t}(x)-1]\exp[g'_{s,t}(x)]+1.
\end{equation*}
Since $h_{s,t}(c)=0$, $[g_{s,t}''(x)]^2+g_{s,t}'''(x)>0$ by Theorem~\ref{byproduct-1} and
\begin{equation*}
h'_{s,t}(x)=-g_{s,t}(x)\bigl\{[g''_{s,t}(x)+g'''_{s,t}(x)]^2\bigr\} \exp[g'_{s,t}(x)]\lessgtr0
\end{equation*}
for $|s-t|\lessgtr1$ and $x\ne c$, we obtain $(x-c)h_{s,t}(x)\lessgtr0$ and hence $f'_{s,t}(x)\lessgtr0$ for $|s-t|\lessgtr1$ and $x\ne c$.
\end{proof}

\end{document}